\documentclass[11pt,reqno]{amsart}
\usepackage{graphicx}
\usepackage{amsmath,amsthm,amssymb,mathtools,xcolor}
\usepackage{fullpage}

\newcommand{\Q}{\mathbb{Q}}
\newcommand{\Z}{\mathbb{Z}}
\newcommand{\Gal}{\operatorname{Gal}}
\newcommand{\GL}{\operatorname{GL}}
\newcommand{\SL}{\operatorname{SL}}
\newcommand{\Aut}{\operatorname{Aut}}

\newcommand{\oQ}{\overline{\Q}}

\newtheorem{theorem}{Theorem}

\newtheorem{proposition}[theorem]{Proposition}
\newtheorem{corollary}[theorem]{Corollary}

\newcommand{\cI}{\mathcal{I}}
\newcommand{\tors}{\textrm{tors}}

\newcommand{\CM}{\textrm{CM}}
\newtheorem*{conjecture*}{Conjecture}

\newcommand{\cF}{\mathcal{F}}

\theoremstyle{definition}

\title{Uniform Polynomial Bounds on Torsion From Rational Geometric Isogeny Classes}

\author{Abbey Bourdon}
\address{Wake Forest University, Winston-Salem, NC 27104, USA}
\email{bourdoam@wfu.edu}
\author{Tyler Genao}
\address{The Ohio State University, Columbus, OH 43210, USA}
\email{genao.5@osu.edu}

\begin{document}
\begin{abstract}
In 1996, Merel \cite{Mer96} showed there exists a function $B\colon \Z^+\rightarrow \Z^+$ such that for any elliptic curve $E/F$ defined over a number field of degree $d$, one has the torsion group bound $\# E(F)[\textrm{tors}]\leq B(d)$.  Based on subsequent work, it is conjectured that one can choose $B$ to be \textit{polynomial} in the degree $d$. In this paper, we show that such bounds exist for torsion from the family $\cI_\Q$ of elliptic curves which are geometrically isogenous to at least one rational elliptic curve. More precisely, we show that for each $\epsilon>0$, there exists $c_\epsilon>0$ such that for any elliptic curve $E/F\in \cI_\Q$, 
one has
\[
\#  E(F)[\tors]\leq c_\epsilon\cdot [F:\Q]^{3+\epsilon}.
\]
This generalizes work of the second author \cite{Gen24} for elliptic curves within a fixed rational geometric isogeny class. For the family of elliptic curves with rational $j$-invariant, we also obtain bounds which improve those of Clark and Pollack \cite{CP18}. In this case, our bounds on the exponent of $E(F)[\tors]$ are optimal if one does not exclude elliptic curves with complex multiplication.
\end{abstract}
\maketitle

\section{Introduction}
For an elliptic curve $E$ defined over a number field $F$, a classic theorem of Mordell and Weil shows that the group $E(F)$ of $F$-rational points on $E$ is a finitely generated abelian group. In particular, its torsion group $E(F)[\tors]$, which is the subgroup of points of finite order, is a finite abelian group. 
It has been of great interest within the past half-century to understand $E(F)[\tors]$ \textit{uniformly,} i.e., for varying elliptic curves $E/F$. 
In 1996, Merel \cite{Mer96} proved his celebrated ``strong uniformity" theorem: for each integer $d>0$, there exists a constant $B(d)>0$ such that for all elliptic curves $E/F$ where $[F:\Q]=d$, one has $\#E(F)[\tors]\leq B(d)$. 
Merel's work provided explicit bounds on primes $\ell\mid \# E(F)[\tors]$ in terms of $d$, which Parent \cite{Par99} later improved by showing that if $\ell^n\mid \# E(F)[\tors]$, then $\ell^n\leq 129(5^d-1)(3d)^6$.

By the work above, one can construct an explicit bound function $B\colon \Z^+\rightarrow \Z^+$ such that for any number field $F$ of degree $d$ and for any elliptic curve $E/F$, one has $\# E(F)[\tors]\leq B(d)$. However, even by Parent's work, such a function would be more than exponential in the degree $d$. Thus, the next natural question is the following: \textit{which choice of function $B$ is close to the ``truth" of torsion growth?} One knows the truth for the family of elliptic curves with \textit{complex multiplication} (CM): by work of Clark and Pollack \cite{CP15}, there exists an absolute, effectively computable constant $c_{\CM}\in\Z^+$, such that for all number fields $F/\Q$ of degree $d\geq 3$ and for all CM elliptic curves $E/F$, one has $\# E(F)[\tors]\leq c_{\CM}\cdot d\log\log d$. This is best-possible by work of Breuer \cite{Bre10}.

Despite our current understanding of torsion groups of CM elliptic curves, the non-CM case remains much more mysterious. Though a number of conjectures and questions appear in the literature concerning the truth of $B$ in the non-CM case (see, for example, \cite[$\S3$]{HS99}), one of the least ambitious forms of a general conjecture involves bounds which are \textit{polynomial} in the degree. 
\begin{conjecture*}
There exist absolute constants $b,c>0$ such that for any number field $F$ of degree $d$ and for any elliptic curve $E/F$, one has $\# E(F)[\emph{tors}]\leq c\cdot d^{b}$.
\end{conjecture*}
That such bounds might exist was first suggested in a question of Flexor and Oesterl\'{e} \cite[Remarque 3]{FO90} for elliptic curves with everywhere good reduction, and the result was later established in this case by Hindry and Silverman \cite{HS99}. More recent work of Clark and Pollack \cite[Theorem 1.3]{CP18} implies there exists polynomial bounds on torsion from the family $\cF_\Q$ of elliptic curves with rational $j$-invariant: for any $\epsilon>0$, there exists a constant $c_\epsilon>0$ such that for all elliptic curves $E/F\in \cF_\Q$, one has $\# E(F)[\tors]\leq c_\epsilon\cdot [F:\Q]^{5/2+\epsilon}$. They also give polynomial bounds in several other cases, including for $E/F$ with $j(E)$ lying in a fixed quadratic field that is not imaginary quadratic of class number one; see \cite[Theorem 1.5, 1.6, 1.7]{CP18}.

In this paper, we extend polynomial bounds from $\cF_\Q$ to a significantly larger family of elliptic curves. For two elliptic curves $E_1/F_1$ and $E_2/F_2$, a \textit{geometric isogeny} from $E_1$ to $E_2$ is a non-constant $\oQ$-rational morphism $\phi\colon E_1/\oQ\rightarrow E_2/\oQ$ which fixes base points. For an elliptic curve $E/F$, the collection of all elliptic curves geometrically isogenous to $E$ is called its \textit{geometric isogeny class.} Let us define the family
\[
\cI_\Q:=\{ E/F~:~\exists E'/\Q\textrm{ for which }E\textrm{ is geometrically isogenous to }E'\}.
\]
The family $\cI_\Q$ is essentially the union over all \textit{rational} geometric isogeny classes, which are geometric isogeny classes which contain least one rational elliptic curve, and it has been studied in several previous works \cite{CN21,BourdonNajman2021,Gen23,BourdonRyallsWatson23}.
We see that $\cF_\Q\subseteq \cI_\Q$, and this is a strict containment since $\cI_\Q$ contains elliptic curves whose $j$-invariant has arbitrarily large degree over $\Q$. 

Our first main result is the following, which generalizes \cite[Theorem 1]{Gen24} in the case of rational geometric isogeny classes. Here, we use ``exp" to denote the exponent of a finite commutative group.
\begin{theorem}\label{Thm_UniformPolynomialBoundsOnRationalGeometricIsogenyClasses}
Torsion from $\cI_\Q$ is polynomially bounded. More precisely, for each $\epsilon>0$, there exists $c_\epsilon>0$ such that for any elliptic curve $E/F\in \cI_\Q$, one has
\[
\exp E(F)[\emph{tors}]\leq c_\epsilon\cdot [F:\Q]^{2+\epsilon},
\]
as well as
\[
\#  E(F)[\emph{tors}]\leq c_\epsilon\cdot [F:\Q]^{3+\epsilon}.
\]
\end{theorem}

Theorem \ref{Thm_UniformPolynomialBoundsOnRationalGeometricIsogenyClasses} relies on a bound for torsion from $\cF_\Q$, which is Theorem \ref{Thm_Rational_j} below. This improves the group exponent bound from \cite[Theorem 1.3]{CP18} by a square-root factor of the degree. By \cite[Theorem 6]{CCS13} and its proof, this gives the optimal power bound on the exponent as we are not excluding CM elliptic curves. That is, for any $\epsilon>0$, we show the bound is of the order of $[F:\Q]^{1+\epsilon}$, and by \cite{CCS13} the 1 cannot be replaced with anything smaller.

\begin{theorem}\label{Thm_Rational_j}
For each $\epsilon>0$, there exists a constant $c_\epsilon>0$ such that for any elliptic curve $E/F$ with $j(E) \in \Q$, we have
\[
\exp E(F)[\emph{tors}]\leq c_\epsilon\cdot [F:\Q]^{1+\epsilon},
\]
as well as
\[
\#  E(F)[\emph{tors}]\leq c_\epsilon\cdot [F:\Q]^{3/2+\epsilon}.
\]
\end{theorem}

Our improved bounds for non-CM elliptic curves in $\cF_\Q$ result from a more detailed analysis of the field of definition of torsion points $P\in E$ whose orders are supported on primes $\ell>37$. Here, if $E'/\Q$ has $j(E')=j(E)$, then the image of its mod-$\ell$ Galois representation is either surjective, or is equal to the normalizer of a non-split Cartan subgroup; see work of Furio and Lombardo \cite{FurioLombardo23}. In the first case, contributions to the degree of $\Q(P)$ are essentially as large as possible by \cite[Proposition 5.7]{BELOV}. Otherwise, both \cite[Theorem 3.11]{FurioLombardo23} and \cite[Appendix B]{LeFournLemos21} imply that $E'$ has potentially good reduction, and no canonical subgroup of order $\ell$. This allows us to make use of more refined ramification data due to Smith \cite{Smith21}.

If one were to restrict to only non-CM elliptic curves, it is expected that our results can be significantly strengthened. For example, Hindry and Silverman \cite[$\S3$]{HS99} ask whether there exists a constant $c$ such that for any non-CM elliptic curve $E$ defined over $F$ of degree $d$, one has $\# E(F)[\tors] \leq c \sqrt{d \log \log d}$. Such a bound on torsion from non-CM elliptic curves with rational $j$-invariant would follow from an affirmative answer to Serre's uniformity question; see \cite[$\S4.2$]{Bre10}.

\section*{Acknowledgments}
We thank Pete L. Clark for his many helpful comments on an earlier draft which significantly improved the exposition of this work. In addition, his suggested strengthening of Theorem \ref{Thm_2.8Generalization} led to an improvement in one of our main results; see $\S4$ for details.  

The first author was partially supported by NSF grant DMS-2145270.

\section{Background and Prior Results}

\subsection{Results on Galois representations}\label{SubSect_GalReps} Once and for all, fix an algebraic closure $\oQ$ of $\Q$. Given an algebraic extension $F/\Q$ and an elliptic curve $E/F$, one has a natural action of the absolute Galois group $G_F:=\Gal(\oQ/F)$ on the group $E(\oQ)$: for example, if $E$ is given in Weierstrass form, then the action on points $P=(x,y)\in E(\oQ)$ is given via $\sigma\cdot (x,y):=(\sigma(x),\sigma(y))$. For each positive integer $n$, this restricts to an action by group automorphisms on $E[n]$, the $n$-torsion subgroup of $E$. The corresponding group homomorphism is called the \textit{mod-$n$ Galois representation of $E$,} denoted by
\[
\rho_{E,n}\colon G_F\rightarrow \Aut(E[n]).
\]
Fixing a $(\Z/n\Z)$-basis $\lbrace P,Q\rbrace$ of $E[n]$ gives an isomorphism $\Aut(E[n])\cong \GL_2(\Z/n\Z)$, so the image $\rho_{E,n}(G_F)$ 
is an explicit subgroup of matrices, up to conjugation.

For a number field $F$ and an elliptic curve $E/F$ without complex multiplication (henceforth ``non-CM"), landmark work of Serre \cite{serre72} showed there exists a constant $\ell_0:=\ell_0(E/F)\in\Z^+$ such that for all primes $\ell>\ell_0$, one has that the image $\rho_{E,\ell}(G_F)$ is as large as possible: specifically, $\rho_{E,\ell}(G_F)=\GL_2(\Z/\ell\Z)$. A natural question is whether one can choose a number $\ell_0:=\ell_0(F)$ \textit{which works for all non-CM elliptic curves over $F$;} this is referred to as Serre's uniformity question over $F$. 

Serre's uniformity question over $\Q$ is still open, but has seen significant progress due to the contributions of many mathematicians -- see \cite[\S 1]{FurioLombardo23} for a more complete history.
The next theorem is the best current progress towards answering this question. For an odd prime $\ell\in\Z^+$, fix the least positive integer $\epsilon$ which generates $(\Z/\ell\Z)^\times$. Then recall that the \textit{normalizer of a non-split Cartan subgroup of $\GL_2(\Z/\ell\Z)$} is (a conjugate of) the group
\[
C_{ns}^+(\ell)=\left\lbrace 
\begin{bmatrix}
a&b\epsilon\\
b&a
\end{bmatrix},
\begin{bmatrix}
a&b\epsilon\\
-b&-a
\end{bmatrix}\in \GL_2(\Z/\ell\Z)\right\rbrace.
\]
\begin{theorem}[Furio, Lombardo \cite{FurioLombardo23}] \label{ThmLargePrimeImage}
Let $E/\Q$ be a non-CM elliptic curve, and let $\ell>37$ be prime. Then $\rho_{E,\ell}(G_\Q)$ is either $\GL_2(\Z/\ell\Z)$ or $C_{ns}^+(\ell)$, up to conjugacy.
\end{theorem}

For each prime number $\ell$, the action of $G_F$ 
on $E(\oQ)$ also induces an action on the $\ell$-adic Tate module $T_\ell(E):=\varprojlim E[\ell^k]$, 
and is denoted by
\[
\rho_{E,\ell^\infty}\colon G_F\rightarrow \Aut(T_\ell(E))\cong \GL_2(\Z_\ell),
\]
where $\Z_\ell$ denotes the ring of $\ell$-adic integers.

\subsection{The modular curves $X_0(n)$ and $X_1(n)$}
For an integer $n\in\Z^+$, we have two distinguished subgroups of $\GL_2(\Z/n\Z)$: the first is the \textit{Borel subgroup of $\GL_2(\Z/n\Z)$,} which is the subgroup of upper-triangular matrices, and is denoted $B_0(n)$. We also have the distinguished subgroup
\[
B_1(n):=\left\lbrace \begin{bmatrix}
1&b\\
0&d
\end{bmatrix}\in \GL_2(\Z/n\Z)\right\rbrace.
\]
These subgroups correspond to the \textit{modular curves} $X_0(n)$ and $X_1(n)$, respectively, which can be regarded as algebraic curves over $\Q$. We will often consider \emph{closed points} on these curves, which correspond to $G_\Q$-orbits of points over $\overline{\Q}$. For a closed point $x$, we use $\Q(x)$ to denote its residue field, and $[\Q(x):\Q]$ is called the \emph{degree} of $x$.

The non-cuspidal $\overline{\Q}$-valued points of $X_0(n)$ parametrize elliptic curves $E/\oQ$ with a fixed order $n$ cyclic subgroup $C$. For a closed point $x\in X_0(n)$ corresponding to $E/\overline{\Q}$ with a cyclic subgroup $C\subseteq E(\overline{\Q})$ of order $n$, there exists $E'/\Q(x)$ with $j(E')=j(E)$ for which the group corresponding to $C$ is $\Q(x)$-rational, i.e., is stabilized under the action of $G_{\Q(x)}$.  See \cite[$\S3$]{Cla} for details. Similarly, the non-cuspidal points of $X_1(n)$ parametrize elliptic curves $E/\oQ$ with an order $n$ torsion point $P$. For a closed point $x=[E,P]\in X_1(n)$, there exists $E'/\Q(x)$ with $j(E')=j(E)$ for which the point corresponding to $P$ is $\Q(x)$-rational.
See \cite[p. 274, Proposition VI.3.2]{DR} for details.

When working with \emph{closed} points on modular curves, we will often write $E/F$ for a model of $E/\overline{\Q}$ over a number field $F$ containing $j(E)$ to avoid introducing new letters for the ``same" curve. This is permissible, since closed points correspond to orbits of $\overline{\Q}$-isomorphism classes.

\subsection{Degrees of Points on $X_1(\ell^k)$ for $\ell>37$} The following proposition essentially follows from work of Ejder \cite[Lemma 3.3]{Ejder22}. However, this work relies on \cite{LR16}, which contains an error. We will reprove this result by applying work of Smith \cite{Smith21}.  
\begin{proposition}[Ejder, \cite{Ejder22}]\label{PropDegModCurve}
Let $E/\Q$ be a non-CM elliptic curve, and let $\ell>37$ be prime. For $P \in E(\overline{\Q})$ of order $\ell^k$, let $x=[E,P] \in X_1(\ell^k)$ be the associated closed point on the modular curve. Then
\[
\deg(x)=\frac{1}{2}(\ell^2-1)\ell^{2k-2}=\deg(X_1(\ell^k) \rightarrow X_1(1)).
\]
\end{proposition}
\begin{proof}
If $\rho_{E,\ell}(G_\Q)=\GL_2(\Z/\ell\Z)$, then in fact $\rho_{E,\ell^{\infty}}(G_\Q)=\GL_2(\Z_{\ell})$, and the conclusion holds. So we may assume $\rho_{E,\ell}(G_\Q)=C_{ns}^+(\ell)$ by Theorem \ref{ThmLargePrimeImage}. If $n=1$, a group theory argument shows that $[\Q(P):\Q]=\ell^2-1$; see \cite[Lemma 5.2]{GJNajman20}. Since $[\Q(P):\Q] \mid 2\deg(x)$ and $\deg(x) \leq \deg(X_1(\ell) \rightarrow X_1(1))=\frac{1}{2}(\ell^2-1)$, the claim follows in this case. The general case will follow once we establish that $[\Q(P):\Q(\ell^{k-1} P)]=\ell^{2k-2}$ when $k>1$. 

By \cite[Theorem 3.11]{FurioLombardo23} and \cite[Appendix B]{LeFournLemos21}, we know $E$ has potentially good reduction at $\ell$ and does not have a canonical subgroup of order $\ell$. Let $R\in E(\overline{\Q})$ be a point of order 5, and define $F:=\Q(R)$. Then $E/F$ has good reduction at each prime above $\ell$ by \cite[Theorem 2]{Frey77}. 
Since $E$ has no canonical subgroup at $\ell$, work of Smith \cite[Theorem 1.1]{Smith21} implies
\[
\ell^{2k}-\ell^{2k-2}
\]
divides the ramification index of any prime in $F(P)$ lying above $\ell$, where $P$ has order $\ell^{k}$. Hence $ \ell^{2k}-\ell^{2k-2}$ divides the degree of $F(P)/\Q$. It follows from $\ell>37$ and $[F:\Q]\leq 24$ that $[F(\ell^{k-1}P):\Q]$ is relatively prime to $\ell$, and so
\[
\ell^{2k-2} \mid [F(P):F(\ell^{k-1}P)].
\]
The desired conclusion about $[\Q(P):\Q(\ell^{k-1}P)]$ follows.
\end{proof}

\section{Lower Bounds for Degrees of Points on $X_1(n)$}
The following proposition is a prelude to Theorem \ref{Thm_IsogenyAndTorsionOrderBoundsUpToEpsilon}. It allows us to relate the degrees of points on $X_1(n)$ and $X_0(n)$ associated to a non-CM $\Q$-rational $j$-invariant directly to the level $n$, when $n$ is supported on any set of uniformly large primes.
\begin{proposition}\label{Prop_IsogenyAndTorsionOrderBounds} 
Suppose $n\in \Z^+$ is divisible only by primes $\ell >37$. Let $E/\Q$ be a non-CM elliptic curve. Then for any $x=[E,P] \in X_1(n)$ and $y=[E, \langle P \rangle] \in X_0(n)$, one has
\[
\deg(x) \geq \frac{1}{2}\cdot \frac{1}{24^k} \cdot n^2 \cdot \prod_{\ell \mid n} \left( 1 - \frac{1}{\ell^2}\right)
\]
and
\[
\deg(y) \geq \frac{1}{24^k} \cdot n \cdot \prod_{\ell \mid n} \left( 1+ \frac{1}{\ell}\right),
\]
where $k$ denotes the number of primes $\ell$ dividing $n$ for which $\rho_{E,\ell}(G_\Q)\neq \emph{GL}_2(\Z/\ell\Z)$.
\end{proposition}

\begin{proof} 
We will first prove the lower bound on $\deg(x)$. Let
\[
S_E \coloneqq \{2,3\} \cup \{ \ell : \rho_{E,\ell^{\infty}}(G_{\Q}) \not\supset \SL_2(\mathbb{Z}_{\ell})\} \cup \{5, \text{ if } \rho_{E,5^{\infty}}(G_{\Q}) \neq \GL_2(\mathbb{Z}_{5})\}.
\]
Suppose $n=\prod \ell_i^{a_i} \cdot \prod q_j^{b_j}$ where $\ell_i,q_j$ are primes with $\ell_i \in S_E$ and $q_j \notin S_E$. Define $n_1 :=\prod \ell_i^{a_i}$ and $ n_2:=\prod q_j^{b_j}$. By \cite[Proposition 5.7]{BELOV},
\[
\deg(x)=\deg(f)\cdot \deg(f(x)),
\] where $f\colon  X_1(n) \rightarrow X_1(n_1)$ is the natural map. Using the standard formulas for $\deg(f)$ we find
\[
\deg(x)=\deg(f(x)) \cdot c_f \cdot n_2 ^2 \cdot \prod \left(1-\frac{1}{q_i^2} \right),
\]
where $c_f =1$ if $n_1>2$ and $c_f=\frac{1}{2}$ otherwise (see e.g. \cite[p.66]{modular}). If $n_1 \leq 2$, then the assumptions on $n$ force $n_1=1$, and $\deg(x)$ has the desired degree. So assume $n_1>2$. It remains to consider the case where $n$ is supported entirely on primes in $S_E$.

So let us assume that $x=[E,P] \in X_1(n)$ where $n=\prod \ell_i^{a_i}$ with $\ell_i\in S_E$. We will write $[E,P_i] \in X_1(\ell_i^{a_i})$ for the image of $x$ under the natural map. Let $R\in E(\overline{\Q})$ be a point of order 5, and define $F:=\Q(R)$; observe that $[F:\Q]\leq 24$. Then as in the proof of Proposition \ref{PropDegModCurve}, the curve $E/F$ has good reduction at each prime above $\ell$ and
\[
\ell_i^{2a_i}-\ell_i^{2a_i-2}
\]
divides the ramification index of any prime in $F(P_i)$ lying above $\ell_i$.
Thus the ramification index of $F( P_i)/F$ at primes above $\ell_i$ is at least
\[
\frac{\ell_i^{2a_i}-\ell_i^{2a_i-2}}{24}.
\]

By the N\'{e}ron-Ogg-Shafarevich criterion, the extension $F(P_i)/F$ is unramified at $\ell_j$ for $j\neq i$. 
Thus the extension $F(P)/F$, which contains the compositum of the field extensions $F( P_i)$, has degree at least
\[
\prod_{i} \frac{\ell_i^{2a_i}-\ell_i^{2a_i-2}}{24}.
\]
Since $\deg(x) \geq \frac{1}{2}\cdot [F(P):F]$, this provides the desired lower bound for $\deg(x)$.

From this, we can deduce a bound on $\deg(y)$. Since 
\[
[\Q(x):\Q(y)] \leq \frac{\varphi(n)}{2},
\] 
the lower bounds on $\deg(x)$ imply
\begin{align*}
[\Q(y):\Q]&=\frac{[\Q(x):\Q]}{[\Q(x):\Q(y)]} \geq 
\frac{1}{24^k}  \cdot n \cdot \prod_{\ell \mid n} \left( 1 + \frac{1}{\ell}\right). \qedhere
\end{align*}
\end{proof}

\section{Polynomial Bounds for Points of Finite Support}
In this section, we will prove the following ``finite support" theorem, which is an isogenous variant of \cite[Theorem 2.8]{CP18}. 
A key distinction is that our implicit constant depends only on the degree of a single fixed $j$-invariant in the geometric isogeny class. We thank Pete L. Clark for the suggestion to bound all of $\#E(F)[\tors]$ here, instead of just the group exponent. This resulted in an improved bound for $\#E(F)[\tors]$ in Theorem \ref{Thm_Rational_j}.
\begin{theorem}\label{Thm_2.8Generalization}
Fix an integer $d_0\in\Z^+$ and a finite set of primes $S\subseteq \Z^+$. Then there exists a constant $c:=c(d_0, S)>0$ such that the following holds. Let $n\in\Z^+$ be any integer whose prime divisors lie in $S$. Suppose $E/F$ is a non-CM elliptic curve geometrically isogenous to an elliptic curve with $j$-invariant $j_0$ of degree $[\Q(j_0):\Q]=d_0$. If $\exp E(F)[\emph{\tors}] =n$, then
\[
\#E(F)[\emph{tors}]\leq c\cdot [F:\Q]^{1/2}.
\]
Moreover, for any cyclic subgroup $C\subseteq E(\oQ)$ of order $n$ that is $F$-rational, one also has
\[
n\leq c\cdot [F:\Q].
\]
\end{theorem}

\begin{proof}
By \cite[Lemma A.4]{CN21}, there exists an elliptic curve ${E_0}$ defined over $F_0:=\Q(j_0)$ whose $j$-invariant $j({E_0})=j_0$, for which $E$ and ${E_0}$ are $L$-rationally isogenous with $L/FF_0$ at most quadratic. 
As a consequence of e.g. \cite[Proposition 2.1.1]{Gre12} or \cite[Corollary 4]{Gen24}, one has for all primes $\ell\in\Z^+$ that 
\begin{equation}\label{EqnlAdicDivTorsionPointDeg}
[\GL_2(\Z_\ell):\rho_{E,\ell^\infty}(G_{L})]=[\GL_2(\Z_\ell):\rho_{{E_0},\ell^\infty}(G_{L})].
\end{equation}

We will prove the bound on $\#E(F)[\tors]$ first. Suppose $E(F)[\ell^{\infty}] \cong \Z/\ell^a\Z \times \Z/\ell^b\Z$ for some $\ell \in S$ and $a \leq b$.
If $a=0$, it follows that $\rho_{E,\ell^b}(G_F)\subseteq B_1(\ell^b)$ once we fix an appropriate basis. Otherwise, there exists a basis of $E[\ell^b]$ such that
\[
\rho_{E,\ell^b}(G_F)\subseteq \left\lbrace \begin{bmatrix}
1&\ell^a x\\
0&1+\ell^a y\end{bmatrix}\Bigg\rvert\, 0 \leq x,y < \ell^{b-a}\right\rbrace.
\] In either case, the index of $\rho_{E,\ell^b}(G_F)$ in $\GL_2(\Z/\ell^b\Z)$ is divisible by $\ell^{\max(0,2a+2b-3)}(\ell^2-1)$, from which it follows that 
\[
\ell^{\max(0,2a+2b-3)}(\ell^2-1) \mid [\GL_2(\Z_\ell):\rho_{E,\ell^\infty}(G_{F})].
\]
Additionally, as the fixed field of the $\ell$-adic representation $\rho_{E_0,\ell^\infty}\colon G_{F_0}\rightarrow \GL_2(\Z_\ell)$ is $F_0(E_0[\ell^\infty])$, we also see that
\[
[\GL_2(\Z_\ell):\rho_{E_0,\ell^\infty}(G_L)]=[\GL_2(\Z_\ell):\rho_{E_0,\ell^\infty}(G_{F_0})]\cdot [L\cap F_0(E_0[\ell^\infty]):F_0].
\]
Combining these facts with \eqref{EqnlAdicDivTorsionPointDeg}, we deduce that
\[
\ell^{\max(0,2a+2b-3)}(\ell^2-1)\mid [\GL_2(\Z_\ell):\rho_{{E_0},\ell^\infty}(G_{F_0})]\cdot [L:F_0].
\]

By the strong uniform $\ell$-adic Arai theorem \cite[Theorem 2.3.a]{CP18}, there exists a constant $a(d_0, \ell)\in\Z^+$ \textit{which depends only on $d_0$ and $\ell$} for which the $\ell$-adic valuation of $[\GL_2(\Z_\ell):\rho_{{E_0},\ell^\infty}(G_{F_0})]$ is at most $a(d_0,\ell)$. Since $[L:\Q]$ divides $2\cdot d_0! \cdot [F:\Q]$, it follows that
\begin{equation}\label{EqnEllAdicDivUsingArai}
\ell^{\max(0, 2a+2b-3-a(d_0,\ell))}\mid [L:F_0]\mid 2(d_0-1)!\cdot [F:\Q].
\end{equation}

Compiling the divisibilities from \eqref{EqnEllAdicDivUsingArai} across all $\ell\in S$, we conclude that
\[
(\# E(F)[\tors])^2\mid 2\prod_{\ell\in S}\ell^{3+a(d_0,\ell)}\cdot (d_0-1)!\cdot [F:\Q].
\]
The result then follows by taking $c:=\sqrt{2\prod_{\ell\in S}\ell^{3+a(d_0,\ell)}\cdot (d_0-1)!}$.

The proof for the cyclic subgroup order bound follows from the proof for the torsion bound, \textit{mutatis mutandis.} The key difference is that if $E$ has an $F$-rational cyclic $\ell^b$-isogeny, then up to conjugation one has $\rho_{E,\ell^b}(G_{F})\subseteq B_0(\ell^b)$, where $[\GL_2(\Z/\ell^b\Z):B_0(\ell^b)]=\ell^{b-1}(\ell+1)$. Thus, similar work shows that if $E$ has an $F$-rational cyclic $n$-isogeny, 
then one has 
\[
n\mid 2\prod_{\ell\in S}\ell^{1+a(d_0,\ell)}\cdot (d_0-1)!\cdot [F:\Q],
\]
and so one can take the constant $c:=2\prod_{\ell\in S}\ell^{1+a(d_0,\ell)}\cdot (d_0-1)!$ for the bound.
\end{proof}

\section{The Proof of Theorem \ref{Thm_Rational_j}}
The next theorem gives upper bounds on the orders of torsion points and cyclic subgroups of $
\Q$-rational non-CM elliptic curves, in terms of the degrees of their fields of definition. 
This improves the bound in \cite[Theorem 1.3]{CP18} by at least a square-root factor of the degree. After proving this, we will prove Theorem \ref{Thm_Rational_j} as a corollary.

\begin{theorem}\label{Thm_IsogenyAndTorsionOrderBoundsUpToEpsilon}
For each $\epsilon>0$, there exists a constant $c_\epsilon>0$ such that for any non-CM elliptic curve $E/\Q$, for any integer $n\in\Z^+$ and for any torsion point $P\in E(\oQ)$ of order $n$, setting $C:=\langle P\rangle$ one has both
\[
n<c_\epsilon\cdot [\Q(P):\Q]^{1+\epsilon}~~\textrm{  and  }~~ n<c_\epsilon\cdot [\Q(C):\Q]^{2+\epsilon}.
\]
If $n$ is supported on the set of primes $\lbrace \ell>37\rbrace$, one instead has the sharper bounds
\[
n<c_\epsilon\cdot [\Q(P):\Q]^{1/2+\epsilon}~~\textrm{  and  }~~ n<c_\epsilon\cdot [\Q(C):\Q]^{1+\epsilon}.
\]
\end{theorem}
\begin{proof}
Assume that $n>1$; let us write
\[
n=A\cdot B,
\]
where $A\geq 1$ is divisible only by primes $\ell\leq 37$, and $B:=\prod_{i=1}^k \ell_i^{b_i}$ where each prime $\ell_i>37$. 

We will prove the torsion point order bound first. 
Let $P_A:=BP$ and $P_B:=AP$. By Theorem \ref{Thm_2.8Generalization}, we have 
\begin{equation}\label{Eqn_TorsionBoundOnUniformlySupportedPart}
A\leq c\cdot [\Q(P_A):\Q]^{1/2}
\end{equation}
for some absolute constant $c\in\Z^+$. Thus, it will suffice to bound $B$ in terms of $d:=[\Q(P_B):\Q]$. We will do this with an analysis similar to that in \cite[\S 3.2]{CP18}. 

Assume that $B>1$. Then by Proposition \ref{Prop_IsogenyAndTorsionOrderBounds}, we have
\begin{equation}
\label{Eqn_WeakerBoundOnTorsionPointOrder}
d>\frac{1}{2}\cdot \frac{1}{24^k}\cdot B\varphi(B)=\frac{1}{2}\prod_{i=1}^k \frac{\ell_i^{2b_i-1}(\ell_i-1)}{24}.
\end{equation}
Fix $0<\epsilon<\frac{1}{2}$, and choose any real number $Z>0$ for which $\log_{Z/24}(24)<\epsilon$. 
We will partition the prime power divisors of $B$ in the following way. First, we claim there are at most $\log_{Z/24}(2d)$ indices $i$ with $Z< \ell_i^{2b_i-1}(\ell_i-1)$. If this were not the case, then we would have by \eqref{Eqn_WeakerBoundOnTorsionPointOrder} that
\begin{align*}
d&> \frac{1}{2}\cdot \prod_{i=1}^k \frac{\ell_i^{2b_i-1}(\ell_i-1)}{24}\\
&\geq \frac{1}{2}\cdot \prod_{\ell_i^{2b_i-1}(\ell_i-1)> Z}\frac{\ell_i^{2b_i-1}(\ell_i-1)}{24}\\
&\geq \frac{1}{2}\cdot \prod_{\ell_i^{2b_i-1}(\ell_i-1)> Z}\frac{Z}{24}\\
&\geq \frac{1}{2}\cdot \left(\frac{Z}{24}\right)^{\log_{Z/24}(2d)}\\
&= d,
\end{align*}
which is absurd. On the other hand, for any index $i$ with $Z\geq \ell_i^{2b_i-1}(\ell_i-1)$, one also has $\ell_i\leq  Z$. We conclude that the number $k$ of distinct prime factors of $B$ satisfies
\begin{equation}\label{Eqn_BoundOnNumberOfPrimeFactors}
k\leq \log_{Z/24}(2d)+\pi(Z),
\end{equation}
where $\pi\colon \Z^+\rightarrow \Z^+$ denotes the prime counting function. 

Using inequalities \eqref{Eqn_WeakerBoundOnTorsionPointOrder} and \eqref{Eqn_BoundOnNumberOfPrimeFactors}, we make the following calculations:
\begin{align*}
B\varphi(B)&<2d\cdot 24^k&&\textrm{By Equation }\eqref{Eqn_WeakerBoundOnTorsionPointOrder}\\
&\leq 2d\cdot 24^{\log_{Z/24}(2d)+\pi(Z)}&&\textrm{By Equation \eqref{Eqn_BoundOnNumberOfPrimeFactors}}\\
&=2d\cdot 24^{\log_{Z/24}(2d)}\cdot 24^{\pi(Z)}\\
&=2d\cdot (2d)^{\log_{Z/24}(24)}\cdot 24^{\pi(Z)}&&\textrm{By }a^{\log_x(b)}=b^{\log_x(a)}\\
&=2^{1+\log_{Z/24}(24)}24^{\pi(Z)}\cdot d^{1+\log_{Z/24}(24)}\\
&<2^{1+\log_{Z/24}(24)}24^{\pi(Z)}\cdot d^{1+\epsilon}&&\textrm{By our initial assumption that }\log_{Z/24}(24)<\epsilon.
\end{align*}
Therefore, we deduce that 
\begin{equation}\label{Eqn_BetterBoundOnTorsionPointOrder}
B\varphi(B)<c_{1,\epsilon}\cdot d^{1+\epsilon},
\end{equation}
where $c_{1,\epsilon}:=2^{1+\log_{Z/24}(24)}24^{\pi(Z)}$ depends only on $\epsilon$ since $Z$ depends only on $\epsilon$. 

Next, by \cite[Theorem 327]{HW08} there exists a constant $0<b_\epsilon<1$ such that for all $n\in\Z^+$ one has 
\[
b_\epsilon\cdot n^{1-\epsilon}<\varphi(n).
\]
Combining this with \eqref{Eqn_BetterBoundOnTorsionPointOrder}, we deduce that
\[
B^{2-\epsilon}<b_\epsilon ^{-1} c_{1,\epsilon}\cdot d^{1+\epsilon},
\]
which implies that
\[
B<(b_\epsilon^{-1}c_{1,\epsilon})^{1/(2-\epsilon)}\cdot d^{(1+\epsilon)/(2-\epsilon)}.
\]
Since $\epsilon<\frac{1}{2}$, we conclude that
\begin{equation}\label{Eqn_BoundOnTorsionSupportedAwayFrom37}
B<c_\epsilon\cdot d^{1/2+\epsilon}
\end{equation}
where $c_\epsilon:=(b_\epsilon^{-1}c_{1,\epsilon})^{1/(2-\epsilon)}>0$.
Combining this with the bound for $A$ in \eqref{Eqn_TorsionBoundOnUniformlySupportedPart} finishes the proof for the torsion point order bound.

Next, we will prove the analogous upper bound for the order $n$ of the subgroup $C:=\langle P\rangle$. We define $C_A:=\langle P_A\rangle$ and $C_B:=\langle P_B\rangle$ to be the unique subgroups of orders $A$ and $B$, respectively. By Theorem \ref{Thm_2.8Generalization}, we have
\begin{equation}\label{Eqn_IsogenyBoundOnUniformlySupportedPart}
A\leq c\cdot [\Q(C_A):\Q]
\end{equation}
for some absolute constant $c\in\Z^+$. 

Fix $0<\epsilon<\frac{1}{10}$. Then we know by 
\eqref{Eqn_BoundOnTorsionSupportedAwayFrom37} that
\[
B<c_\epsilon\cdot [\Q(P_B):\Q]^{1/2+\epsilon}
\]
for some $c_\epsilon>0$. 
Let us make some observations about this inequality. We will set $d:=[\Q(C_B):\Q]$:
\begin{align*}
B&<c_\epsilon\cdot [\Q(P_B):\Q]^{1/2+\epsilon}\\
&=c_\epsilon\cdot [\Q(P_B):\Q(C_B)]^{1/2+\epsilon}\cdot d^{1/2+\epsilon}\\
&\leq c_\epsilon\cdot \varphi(B)^{1/2+\epsilon}\cdot d^{1/2+\epsilon}.
\end{align*}
We thus have
\[
\frac{B}{\varphi(B)^{1/2+\epsilon}}<c_\epsilon\cdot d^{1/2+\epsilon}.
\]
Since $\varphi(B)\leq B$, we get $\frac{1}{B^{1/2+\epsilon}}\leq\frac{1}{\varphi(B)^{1/2+\epsilon}}$, which means we have the inequality
\[
B^{1/2-\epsilon}<c_\epsilon\cdot d^{1/2+\epsilon}.
\]
Thus,
\[
B<c_{\epsilon}^{1/(1/2-\epsilon)}\cdot d^{\frac{1/2+\epsilon}{1/2-\epsilon}}.
\]
From $\epsilon<\frac{1}{10}$ we have $\frac{1/2+\epsilon}{1/2-\epsilon}<1+5\epsilon$, and so our inequality becomes
\[
B<c_\epsilon'\cdot d^{1+5\epsilon}
\]
where $c_\epsilon':=c_{\epsilon}^{1/(1/2-\epsilon)}$. Adjusting $\epsilon$ as necessary,
we can assume we have the bound
\[
B<c_{\epsilon}'\cdot d^{1+\epsilon}.
\]
Combining this bound with \eqref{Eqn_IsogenyBoundOnUniformlySupportedPart} proves our cyclic subgroup order bound.
\end{proof}
Here is Theorem \ref{Thm_Rational_j}, stated as a corollary of Theorem \ref{Thm_IsogenyAndTorsionOrderBoundsUpToEpsilon}.
\begin{corollary}\label{Theorem2Cor}
For each $\epsilon>0$, there exists a constant $c_\epsilon>0$ such that for any elliptic curve $E/F$ with $j(E) \in \Q$, we have
\[
\exp E(F)[\emph{tors}]\leq c_\epsilon\cdot [F:\Q]^{1+\epsilon},
\]
as well as
\[
\#  E(F)[\emph{tors}]\leq c_\epsilon\cdot [F:\Q]^{3/2+\epsilon}.
\]
\end{corollary}

\begin{proof}
    By results in the CM case due to Clark and Pollack \cite[Theorem 1]{CP15}, we may assume $E$ is non-CM. A quick twisting argument shows that the bounds on $\exp E(F)[\tors]$ follow from Theorem \ref{Thm_IsogenyAndTorsionOrderBoundsUpToEpsilon}, so it suffices to prove the stated bound on $\#E(F)[\tors]$. 
    
   Suppose $E(F)[\tors] \cong \Z/m\Z \times \Z/n\Z$ for $m \mid n$. Let $d \coloneqq [F:\Q]$. Let $m=m_1m_2$ and $n=n_1n_2$, where $m_1, n_1$ are supported on primes $\ell \leq 37$ and $m_2,n_2$ are supported on primes $\ell >37$. By Theorem \ref{Thm_2.8Generalization}, we have $m_1n_1 \leq c \cdot d^{1/2}$ for some absolute constant $c>0$. Now, let $\epsilon>0$. Then Theorem \ref{Thm_IsogenyAndTorsionOrderBoundsUpToEpsilon} shows that $n_2<c_{1,\epsilon} \cdot d^{1/2+\epsilon}$ for some constant $c_{1,\epsilon}>0$. Since $m_2n_2 \mid n_2^2$, we have
\[
m_2n_2 < c_{1,\epsilon}^2 \cdot d^{1+\epsilon}
\]
(adjusting $\epsilon$ as necessary).
Putting this together gives
\[
mn=m_1n_1m_2n_2< c_{\epsilon}\cdot d^{3/2+\epsilon}
\]
where $c_\epsilon:=c\cdot c_{1,\epsilon}^2$.
\end{proof}
\section{The Proof of Theorem \ref{Thm_UniformPolynomialBoundsOnRationalGeometricIsogenyClasses}}
Let $E/F$ be an elliptic curve in $\cI_\Q$, and set $d\coloneqq [F:\Q]$. If $E$ has CM, then by \cite[Theorem 1]{CP15} we have $\# E(F)[\tors]\leq c\cdot d\log\log d$ for some absolute, effectively computable constant $c>0$ when $d\geq 3$. Thus, we may assume that $E$ is non-CM. Then by \cite[Lemma A.4]{CN21}, there exists an elliptic curve $E'/\Q$ and an extension $K/F$ of degree at most $2$ for which $E$ and $E'$ are $K$-rationally isogenous. Let us write $\phi\colon E\rightarrow E'$ for such an isogeny; we may assume that $\phi$ is cyclic. 

We will first prove polynomial bounds on $N:=\exp E(F)[\tors]$. Let us write
\[
N=Mm
\]
for some $M,m\in \Z^+$, where for primes $\ell\in \Z^+$ one has $\ell\mid M$ if and only if $\ell\mid N$ and $\ell\leq 37$.
By Theorem \ref{Thm_2.8Generalization},
there exists an absolute constant $c_1\in \Z^+$ for which 
\begin{equation}\label{Eqn_Exponent_UniformlySupportedPart}
M\leq c_1\cdot [F:\Q]^{1/2}.
\end{equation}
Thus, it suffices to polynomially bound $m$.

Let $P \in E(K)$ be a point of order $m$. Then $\phi(P)\in E'(K)$ has order $m_1$ for some $m_1 \mid m$. Let us write $m=m_1m_2$. Fixing an $\epsilon>0$, by Theorem \ref{Thm_IsogenyAndTorsionOrderBoundsUpToEpsilon} we have
\begin{equation}\label{Eqn_Exponent_AwayFrom37_Squarefree_CoprimeToIsogenyDegree}
m_1< c_{1,\epsilon}\cdot [K:\Q]^{1/2+\epsilon}
\end{equation}
for some constant $c_{1,\epsilon}>0$ which depends only on $\epsilon$. Next, since 
\[
\mathcal{O}=m_1\phi(P)=\phi(m_1P),
\]
it follows that $m_1P \in \ker(\phi)$. Thus $m_2 \mid \#\ker(\phi)=\#\ker(\widehat{\phi})$, where $\widehat{\phi}\colon E' \rightarrow E$ denotes the dual isogeny. Since $\widehat{\phi}$ is also cyclic and $K$-rational, $\ker(\widehat{\phi})$ contains a unique $K$-rational cyclic subgroup of order $m_2$. This implies by Theorem \ref{Thm_IsogenyAndTorsionOrderBoundsUpToEpsilon} that
\begin{equation}\label{Eqn_Exponent_AwayFrom37_Squarefree_DividesIsogenyDegree}
m_2<c_{2,\epsilon}\cdot [K:\Q]^{1+\epsilon}
\end{equation}
for some constant $c_{2,\epsilon}$ which depends only on $\epsilon$. 

We combine the bounds given in \eqref{Eqn_Exponent_UniformlySupportedPart}, \eqref{Eqn_Exponent_AwayFrom37_Squarefree_CoprimeToIsogenyDegree} and \eqref{Eqn_Exponent_AwayFrom37_Squarefree_DividesIsogenyDegree}, and conclude that (noting $[K:\Q]\mid 2[F:\Q])$ for each $\epsilon>0$, there exists a constant $c_\epsilon>0$, which depends only on $\epsilon$, such that
\begin{equation}\label{Eqn_GroupExponentBound}
\exp E(F)[\tors]<c_\epsilon\cdot [F:\Q]^{2+\epsilon}.
\end{equation}
This proves our group exponent bound. 

Next we will prove our polynomial bound on $\#E(F)[\tors]$, following \cite[$\S3.3$]{CP18}. Writing $N:=\exp E(F)[\tors]$, we have $E(F)[\tors]\cong \Z/d\Z\times\Z/N\Z$ for some $d\mid N$ where $d>0$. Since $E$ has full $d$-torsion over $F$, we find that $F$ contains a primitive $d$'th root of unity by the Weil pairing \cite[Corollary III.8.1.1]{Sil09}, and so $\varphi(d)\mid [F:\Q]$. By \cite[Theorems 327 and 328]{HW08}, this implies that $d\leq b_\epsilon\cdot [F:\Q]^{1+\epsilon}$ for some $b_\epsilon>0$. 
Since $\# E(F)[\tors]=dN$, we combine this with \eqref{Eqn_GroupExponentBound} and conclude that
\[
\# E(F)[\tors]<C_{\epsilon}\cdot [F:\Q]^{3+\epsilon}
\]
where $C_\epsilon:=b_\epsilon\cdot c_\epsilon$.
\bibliographystyle{alpha}
\bibliography{bibfile}

\end{document}